%% file: template.tex
\documentclass{article}

\usepackage{arxiv}

\usepackage[utf8]{inputenc} 
\usepackage[T1]{fontenc}    
\usepackage{hyperref}       
\usepackage{url}            
\usepackage{booktabs}       
\usepackage{amsfonts}       
\usepackage{nicefrac}       
\usepackage{microtype}      
\usepackage{lipsum}		
\usepackage{graphicx}
\usepackage{doi}
\usepackage[backend=biber]{biblatex}
\input{shared}

\title{On the Time Derivative of the KL Divergence for a Generalized Langevin Annealing Scheme}

\author{Andreas Habring\thanks{Institute of Visual Computing, Graz University of Technology, Austria (\href{andreas.habring@tugraz.at}{andreas.habring@tugraz.at}).}}

\renewcommand{\shorttitle}{\textit{arXiv} Template}

\hypersetup{
pdftitle={Habring_timederivative_KL},
pdfsubject={q-bio.NC, q-bio.QM},
pdfauthor={},
pdfkeywords={},
}

\begin{document}

\maketitle

\begin{abstract}
Consider the Langevin diffusion process $\d X_t = \nabla \log p_t(X_t)\d t + \sqrt{2}\d W_t$ guided by the time-dependent probability density $p_t(x)$. Let $q_t$ be the density of $X_t$.
Recently, in \cite{chehab2025provable,vempala2019rapid} in order to analyze convergence in the Kullback-Leibler divergence, the authors have made use of the time derivative of $t\mapsto \KL(q_t|p_t)$ without investigating in detail when such a derivative exists. In this short manuscript we provide a rigorous derivation of the quantity $\frac{\d}{\d t}\KL(q_t|p_t)$.
\end{abstract}
\keywords{Langevin diffusion, stochastic differenctial equations, partial differential equations, sampling.}

\section{Introduction}
Let $(X_t)_t$ be a Langevin diffusion process in $\R^d$ described by the \gls{sde}
\begin{equation}\label{eq:sde}
    \d X_t = \nabla\log p_t(X_t)\d t + \d W_t
\end{equation}
where $(W_t)_t$ denotes standard Brownian motion and $(p_t)_t$ is the Lebesgue density of a time-dependent probability distribution $\pi_t$. The process \eqref{eq:sde} can be understood a Langevin diffusion with a moving target, \cf~\cite{chehab2025provable}. We assume that the family $p_t$ satisfies the following properties.
\begin{assumption}\label{ass}\
    \begin{itemize}
        \item Differentiability: $p_t(x)\in C^1(\R^d\times (0,\infty))$.
        \item Uniform Lipschitz continuity: $\nabla p_t(x)$ is Lipschitz continuous with respect to $x$ uniformly in $t$, that is, there exists $L>0$ such that 
        \[
            |\nabla p_t(x)-\nabla p_t(y)|\leq L|x-y|,\quad t>0, \; x,y\in \R^d.
        \]
        \item Growth of the time-derivative: There exists $c>0$ such that 
        \[
            |\partial_t p_t(x)|\leq c|x|^2.
        \]
        \item Uniform dissipativity: There exist $\dissi,\dissii>0$ such that
    \[
        -x\cdot \nabla \log p_t(x) \geq \dissi |x|^2 - \dissii,\quad t>0, \; x\in \R^d.
    \]
    \end{itemize}
\end{assumption}
By standard theory under these assumptions~\eqref{eq:sde} admits a unique strong solution and, if $(\mu_t)_t$ denotes the distribution of $(X_t)_t$, it is well known that $\mu_t$ satisfies in the weak sense the \gls{pde}
\begin{equation}\label{eq:FP}
    \partial_t \mu_t = -\nabla\cdot (\mu_t\nabla\log p_t) + \Delta \mu_t.
\end{equation}
By~\cite[6.3.2 Corollary]{bogachev2022fokker} it follows that $\mu_t$ admits a density with respect to the Lebesgue measure, that is, $\mu_t = q_t \d x$ for some $q\in L^1_{loc}(\R^d\times (0,\infty))$. Moreover, $q$ is continuous and strictly positive on $\R^d\times (0,\infty)$ and $q_t\in W^{1,p}_{loc}(\R^d)$ for all $p>d+2$ (see in particular the intro to~\cite[Section 9.4]{bogachev2002uniqueness}). In addition, the (weak) solution to \eqref{eq:FP} is unique (\cf~\cite[Theorem 9.4.8]{bogachev2022fokker} with the trivial choice $V(x) = |x|^2)$ leading to a one-to-one correspondence between the Fokker-Planck equation and the \gls{sde}.

\section{Main result}
Recently, in~\cite{vempala2019rapid,chehab2025provable} authors have investigated the quantity $\frac{\d}{\d t}\KL(\mu_t|\pi_t)$.
Indeed, invoking \eqref{eq:FP}, formal computations yield
\begin{equation}\label{eq:goal eq}
    \begin{aligned}
        \frac{\d}{\d t}\KL(\mu_t|\pi_t) &= \frac{\d}{\d t}\int q_t\log\frac{q_t}{p_t}\d x\\
        & =\int \partial_tq_t\log\frac{q_t}{p_t} \d x+ \int q_t\partial_t\log\frac{q_t}{p_t}\d x\\
        &= -\int q_t|\nabla \log\frac{q_t}{p_t}|^2\d x 
        - \int q_t\partial_t\log p_t\d x.
    \end{aligned}
\end{equation}
It is, however, unclear under which circumstances these computations are, in fact, allowed. The aim of this article is to provide a rigorous proof of \eqref{eq:goal eq}.
\begin{theorem}
    Under \cref{ass}, equation \eqref{eq:goal eq} is satisfied, that is
    \begin{equation}
    \begin{aligned}
        \frac{\d}{\d t}\KL(\mu_t|\pi_t) = -\int q_t|\nabla \log\frac{q_t}{p_t}|^2 
        - \int q_t\partial_t\log p_t\d x.
    \end{aligned}
\end{equation}
\end{theorem}
\begin{proof}
    By \cite[Proposition 7.3.10 Example, 8.2.4 Example, 8.2.5 Proposition]{bogachev2022fokker} we have that for any $0<\tau_1<\tau_2<\infty$ there exist $c_1,c_2>0$ such that 
    \begin{equation}\label{eq:density_exponential}
        \expo{-c_1(|x|^2+1))}\leq q_\tau(x)\leq \expo{-c_2|x|^2},\quad \tau\in [\tau_1,\tau_2],\; x\in \R^d.
    \end{equation}
    Let $\phi\in C^\infty_c(\R^d\times \R)$ be a mollifier with, in particular, $\phi(x,t)=0$ for $|t|\geq 1$, denote $\phi^\epsilon(x,t) = \frac{1}{\epsilon^{d+1}}\phi((x,t)/\epsilon)$ and $q_t^\epsilon=\tilde{q}*\phi^\epsilon(x,t)$ with the convolution in $x$ and $t$ where $\tilde{q}$ is the zero for $t\notin[\tau_1-\epsilon,\tau_2+\epsilon]$ and identical to $q$ otherwise. In particular, we can choose $\epsilon$ sufficiently small so that $\tau_1-\epsilon>0$. It follows that $q^\epsilon_t$ is $C^\infty$and satisfies
    \begin{equation}\label{eq:weak_dt}
        \partial_t q^\epsilon_t = - (q_t\partial_{x_i}\log p_t)*\partial_{x_i}\phi^\epsilon + q_t*\partial_{x_i}\partial_{x_i} \phi^\epsilon,\quad (x,t)\in \R^d\times (\tau_1,\tau_2).
    \end{equation}
    Indeed, for any $\psi\in C^\infty_c(\R^d\times(\tau_1,\tau_2))$ since $q_t$ solves \eqref{eq:FP} weakly we have
    \begin{equation*}
            \begin{aligned}
            \iint &\tilde{q}*\phi^\epsilon(x,t) \partial_t \psi(x,t) \d x \d t \\
            &= \iiiint \tilde{q}_{t-s}(x-y) \phi^\epsilon(y,s) \partial_t \psi(x,t) \d x\d y\d t\d s\\
            &=\iint \phi^\epsilon(y,s)\iint q_{t-s}(x-y)  \partial_t \psi(x,t) \d x\d t \d y \d s\\
            &= \iint \phi^\epsilon(y,s)\iint \left( -q_{t-s}(x-y)\partial_{x_i}\log p_{t-s}(x-y) + \partial_{x_i} q_{t-s}(x-y)\right) \partial_{x_i}\psi(x,t) \d x\d t \d y \d s\\
            &= \iint \partial_{x_i}\psi(x,t) \iint \left(- q_{t-s}(x-y)\partial_{x_i}\log p_{t-s}(x-y) + \partial_{x_i} q_{t-s}(x-y)\right) \phi^\epsilon(y,s)\d y \d x\d t\d s \\
            &= \iint \partial_{x_i}\psi(x,t) \left( -q_t\partial_{x_i}\log p_t + \partial_{x_i} q_t\right) * \phi^\epsilon (x)\d x\d t \\
            &= -\iint \psi(x,t)  \left( -(q_t\partial_{x_i}\log p_t)+ \partial_{x_i} q_t\right)*\partial_{x_i}\phi^\epsilon (x)\d x\d t
        \end{aligned}
    \end{equation*}
    where we used that $\phi^\epsilon(y,s)=0$ for $|s|\geq\epsilon$, $\psi(x,t)=0$ for $t\notin [\tau_1,\tau_2]$ and consequently we can replace $\tilde{q}_{t-s}$ by $q_{t-s}$ since they coincide for ${t-s}\in [\tau_1-\epsilon,\tau_2+\epsilon]$. Moreover, by the exponential decay with respect to $x$ and the compact support in $t$, we have that $q_t^\epsilon\rightarrow \tilde{q}_t$ in $L^p(\R^d\times \R)$ as $\epsilon\rightarrow 0$ for any $p\in [1,\infty)$ and pointwise a.e. by taking an appropriate subsequence. 
    In addition $\partial_{x_i}q_t^\epsilon\rightarrow\partial_{x_i}q_t$ in $L^2(\R^d\times (\tau_1,\tau_2)$ and, again by taking a subsequence, pointwise a.e. (see \cite[Theorem 7.4.1]{bogachev2002uniqueness} for a proof that $\nabla q_t \in L^2(\R^d\times (\tau_1,\tau_2))$). Since $q$ is moreover continuous, it follows that $q_t^\epsilon\rightarrow q_t$ uniformly on any set of the form $K\times (\tau_1,\tau_2)$ with $K\subset{\R^d}$ compact.
    Since $q_t$ satisfies \eqref{eq:density_exponential}\footnote{also for adapted $c_1,c_2$ on the slightly larger interval $[\tau_1-\epsilon,\tau_2+\epsilon]$} we find that
    \begin{equation}\label{eq:growthqeps}
        \begin{aligned}
            q_t^\epsilon(x) \leq \max_{y\in B_\epsilon(x)}q_t(y) \leq \max_{y\in B_\epsilon(x)}\exp\left(-c_2|y|^2\right)= \1_{B_{2\epsilon}(0)}(x) + \1_{B^c_{2\epsilon}(0)}(x) \exp\left(-c_2|x-\frac{x}{|x|}\epsilon|^2\right)\\
            = \1_{B_{2\epsilon}(0)}(x) + \1_{B^c_{2\epsilon}(0)}(x) \exp\left(-c_2|x|^2(1-\frac{\epsilon}{|x|})^2\right)\\
            \leq \1_{B_{2\epsilon}(0)}(x) + \1_{B^c_{2\epsilon}(0)}(x) \exp\left(-c_2|x|^2/4\right)\\
            \leq c\exp\left(-c_2|x|^2/4\right)
        \end{aligned}
    \end{equation}
    where the constants can be chosen uniformly with respect to $\epsilon$ and $t\in[\tau_1,\tau_2]$. Similarly we have 
    \begin{equation}\label{eq:growthqeps2}
        \begin{aligned}
            q_t^\epsilon(x) \geq \exp\left(-c_1(2|x|^2+1)\right)
        \end{aligned}
    \end{equation}
    which, in particular, implies that $q_t^\epsilon$ has finite entropy for all $t$. Let $\chi_R\in C^\infty_c(\R^d)$ be such that $0\leq \chi_R\leq 1$, $\chi_R(x) = 1$ for $x\in B_R(0)$ and $\chi_R(x)=0$ for $x\notin B_{R+1}(0)$. Using the fundamental theorem of calculus, we have\footnote{The time integral is over $(\tau_1,\tau_2)$ which we omit for simpler notation.}
    \begin{equation}\label{eq:AB}
        \begin{aligned}
            \int \chi_R q_{\tau_1}^\epsilon\log\frac{q_{\tau_1}^\epsilon}{\pi_{\tau_1}}\dd x - \int \chi_R q_{\tau_1}^\epsilon\log\frac{q_{\tau_2}^\epsilon}{\pi_{\tau_2}}\dd x &= \iint \chi_R\frac{\d}{\d \tau}\left(q^\epsilon_\tau\log\frac{q^\epsilon_\tau}{\pi_\tau}\right) \d \tau \d x\\
            &= \iint \chi_R\frac{\d}{\d \tau}q^\epsilon_\tau\log\frac{q^\epsilon_\tau}{\pi_\tau} + \chi_Rq^\epsilon_\tau\frac{\d}{\d \tau}\log\frac{q^\epsilon_\tau}{\pi_\tau} \d \tau \d x\\
            &= A + B.
        \end{aligned}
    \end{equation}
    Our goal is in the following to take the limit on both sides of \eqref{eq:AB}, first as $\epsilon\rightarrow 0$ and afterward as $R\rightarrow \infty$.
    \paragraph{Convergence of $B$}
    For the term $B$ we can estimate
    \begin{equation}
        \begin{aligned}
            B &= \iint \chi_R \frac{\pi_\tau\partial_\tau q_\tau^\epsilon-q^\epsilon_\tau \partial_\tau \pi_\tau}{\pi_\tau}\d \tau \d x \\
            &= \iint \chi_R\left(\partial_\tau q^\epsilon_\tau - \chi_Rq^\epsilon_\tau \partial_\tau \log \pi_\tau\right)\d \tau \d x\\
            &= \iint \chi_R\partial_\tau q^\epsilon_\tau \d x\d \tau- \iint \chi_Rq^\epsilon_\tau \partial_\tau \log \pi_\tau\d \tau \d x\\
            &= -\iint ((-q_\tau\partial_{x_i}\log p_\tau)*\phi^\epsilon + q_\tau*\partial_{x_i}\phi^\epsilon) \partial_{x_i} \chi_R\d x\d \tau - \iint \chi_Rq^\epsilon_\tau \partial_\tau \log \pi_\tau\d \tau \d x\\
            &= AA + BB.
        \end{aligned}
    \end{equation}
    As $\epsilon\rightarrow 0$ we then find for $BB$
    \begin{equation}
        \begin{aligned}
            \lim_{\epsilon\rightarrow 0} BB 
            = -\iint \chi_Rq_\tau \partial_\tau \log \pi_\tau\d \tau \d x
        \end{aligned}
    \end{equation}
    due to convergence of $q_\tau^\epsilon$ uniformly on compact sets and the fact that $\chi_Rq_\tau \log \pi_\tau$ is bounded and compactly supported with respect to $x$.
    For $AA$ we obtain using similar arguments that
    \begin{equation}
        \begin{aligned}
            \lim_{\epsilon\rightarrow 0} AA
            &= -\iint ((-q_\tau\partial_{x_i}\log p_\tau) + \partial_{x_i}q_\tau) \partial_{x_i} \chi_R\d x\d \tau
        \end{aligned}
    \end{equation}
    \paragraph{Convergence of $A$}
    Regarding the term $A$ in \eqref{eq:AB} we find
    \begin{equation}
        \begin{aligned}
            A = &\iint \chi_R ((-q_\tau\partial_{x_i}\log p_\tau)*\partial_{x_i}\phi^\epsilon + q_\tau*\partial_{x_i}\partial_{x_i}\phi^\epsilon) \log\frac{q_\tau^\epsilon}{\pi_\tau} \d x \d \tau \\
            = &-\iint ((-q_\tau\partial_{x_i}\log p_\tau)*\phi^\epsilon + q_\tau*\partial_{x_i}\phi^\epsilon) \partial_{x_i}\left( \chi_R \log\frac{q_\tau^\epsilon}{\pi_\tau} \right)\d x\d \tau\\
            =& -\iint ((-q_\tau\partial_{x_i}\log p_\tau)*\phi^\epsilon + q_\tau*\partial_{x_i}\phi^\epsilon) \partial_{x_i}\chi_R \log\frac{q_\tau^\epsilon}{\pi_\tau}\d x\d \tau\\
            &-\iint ((-q_\tau\partial_{x_i}\log p_\tau)*\phi^\epsilon + q_\tau*\partial_{x_i}\phi^\epsilon) \chi_R\partial_{x_i} \log\frac{q_\tau^\epsilon}{\pi_\tau}\d x \d \tau= AA + BB
        \end{aligned}
    \end{equation}
    Again, we consider the limits as $\epsilon\rightarrow 0$. Regarding $AA$ we first separate $AA$ into two terms
    \begin{equation}
        \begin{aligned}
            AA =& -\iint ((-q_t\partial_{x_i}\log p_t)*\phi^\epsilon + q_t*\partial_{x_i}\phi^\epsilon) \partial_{x_i}\chi_R \log q_\tau^\epsilon\d x\\
            &+\iint ((-q_t\partial_{x_i}\log p_t)*\phi^\epsilon + q_t*\partial_{x_i}\phi^\epsilon) \partial_{x_i}\chi_R \log\pi_\tau\d x.
        \end{aligned}
    \end{equation}
    We begin with the more difficult first term
    \begin{equation}
        \begin{aligned}
            \bigg|\iint ((-q_\tau\partial_{x_i}\log p_\tau)*\phi^\epsilon + q_\tau*\partial_{x_i}\phi^\epsilon) \partial_{x_i}\chi_R \log q_\tau^\epsilon \\
            - ((-q_\tau\partial_{x_i}\log p_\tau) + \partial_{x_i}q_\tau) \partial_{x_i}\chi_R \log q_\tau\d x\d \tau\bigg|\\
            \leq \bigg|\iint ((-q_\tau\partial_{x_i}\log p_\tau)*\phi^\epsilon + q_\tau*\partial_{x_i}\phi^\epsilon) \partial_{x_i}\chi_R \log q_\tau^\epsilon \\- ((-q_\tau\partial_{x_i}\log p_\tau) + \partial_{x_i}q_\tau) \partial_{x_i}\chi_R \log q_\tau^\epsilon\d x\d \tau\bigg|\\
            + \bigg|\iint ((-q_\tau\partial_{x_i}\log p_\tau) + \partial_{x_i}q_\tau) \partial_{x_i}\chi_R \log q_\tau^\epsilon \\- ((-q_\tau\partial_{x_i}\log p_\tau) + \partial_{x_i}q_\tau) \partial_{x_i}\chi_R \log q_\tau\d x\d \tau\bigg|\\
            \leq \iint \bigg|(-q_\tau\partial_{x_i}\log p_\tau)*\phi^\epsilon + q_\tau*\partial_{x_i}\phi^\epsilon) - ((-q_\tau\partial_{x_i}\log p_\tau) + \partial_{x_i}q_\tau)\bigg| \big|\partial_{x_i}\chi_R \log q_\tau^\epsilon\big| \d x\d \tau\\
            + \iint \big|((-q_\tau\partial_{x_i}\log p_\tau) + \partial_{x_i}q_\tau) \partial_{x_i}\chi_R \big|\bigg|  \log q_\tau^\epsilon - \log q_\tau\bigg|\d x\d \tau\rightarrow 0
        \end{aligned}
    \end{equation}
    as $\epsilon\rightarrow 0$. For the first integral this follows from $L^1_{loc}$ convergence of $(-q_t\partial_{x_i}\log p_t)*\phi^\epsilon + q_t*\partial_{x_i}\phi^\epsilon)$ and uniform boundedness of $\partial_{x_i}\chi_R \log q_\tau^\epsilon$ where we make use of the estimates \eqref{eq:growthqeps} and \eqref{eq:growthqeps2}. 
    For the second term, to the contrary, we use $L^1$ boundedness of $((-q_t\partial_{x_i}\log p_t) + \partial_{x_i}q_t) \partial_{x_i}\chi_R$ and uniform convergence of $\log q_\tau^\epsilon$ on compact sets. The latter uniform convergence follows from uniform convergence of $q^\epsilon_t$ on sets of the form $K\times (\tau_1,\tau_2)$ for $K\subset\R^d$ compact together with the fact that by \eqref{eq:growthqeps} and \eqref{eq:growthqeps2} the set
    \[
    U\coloneqq\{q^\epsilon_t(x)\;| (x,t)\in K\times (\tau_1,\tau_2), \; 0<\epsilon<\epsilon_{\mathrm{max}}\}
    \]
    satisfies $U\subset [a,b]$ for some $0<a<b<\infty$ and the logarithm restricted to $[a,b]$ is Lipschitz.
    For the second term in $AA$ similar arguments can be repeated using uniform upper and lower bounds of $\log\pi_\tau$ again on compact sets.

    Next, we tackle $BB$. We find 
    \begin{equation}
        \begin{aligned}
            \iint ((-q_\tau\partial_{x_i}\log p_\tau)*\phi^\epsilon + q_\tau*\partial_{x_i}\phi^\epsilon) \chi_R\partial_{x_i} \log\frac{q_\tau^\epsilon}{\pi_\tau}\d x\d \tau\\
            =\iint ((-q_\tau\partial_{x_i}\log p_\tau)*\phi^\epsilon + q_\tau*\partial_{x_i}\phi^\epsilon) \chi_R\partial_{x_i} \log q_\tau^\epsilon \d x\d \tau\\
            -\iint ((-q_\tau\partial_{x_i}\log p_\tau)*\phi^\epsilon + q_\tau*\partial_{x_i}\phi^\epsilon) \chi_R\partial_{x_i} \log \pi_\tau\d x\d \tau
        \end{aligned}
    \end{equation}
    and once again, start with the more difficult first term
    \begin{equation}
        \begin{aligned}
            \iint ((-q_\tau\partial_{x_i}\log p_\tau)*\phi^\epsilon + q_\tau*\partial_{x_i}\phi^\epsilon) \chi_R\partial_{x_i} \log q_\tau^\epsilon \d x\d \tau\\
            =\iint ((-q_\tau\partial_{x_i}\log p_\tau)*\phi^\epsilon + q_\tau*\partial_{x_i}\phi^\epsilon) \chi_R\frac{\partial_{x_i}q_\tau^\epsilon }{q_\tau^\epsilon }\d x\d \tau\\
            =\iint \chi_R\frac{\partial_{x_i}q_\tau^\epsilon }{q_\tau^\epsilon }(-q_\tau\partial_{x_i}\log p_\tau)*\phi^\epsilon + \frac{(\partial_{x_i}q_\tau^\epsilon)^2}{q_\tau^\epsilon }\chi_R\d x\d \tau.
        \end{aligned}
    \end{equation}
    The second term can be understood as the squared $L^2$ norm of $\frac{\partial_{x_i}q_\tau^\epsilon}{\sqrt{q_\tau^\epsilon}}\sqrt{\chi_R}$ which converges if the respective function converges in $L^2$. We show the latter,
    \begin{equation}
        \begin{aligned}
            \|\frac{\partial_{x_i}q_\tau^\epsilon}{\sqrt{q_\tau^\epsilon}}\sqrt{\chi_R} - \frac{\partial_{x_i}q_\tau}{\sqrt{q_\tau}}\sqrt{\chi_R}\|_2\\
            \leq \|\left( \partial_{x_i}q_\tau^\epsilon - \partial_{x_i}q_\tau\right) \frac{\sqrt{\chi_R}}{\sqrt{q_\tau^\epsilon}}\|_2 
            + \|\partial_{x_i}q_\tau \left(\frac{\sqrt{\chi_R}}{\sqrt{q_\tau^\epsilon}} - \frac{\sqrt{\chi_R}}{\sqrt{q_\tau}}\right)\|_2.
        \end{aligned}
    \end{equation}
    Using uniform, strictly positive upper and lower bounds of $q_\tau^\epsilon$ on $B_{R+1}(0)\times (\tau_1,\tau_2)$, the $L^2$ convergence of $\partial_{x_i}q_\tau^\epsilon$ and uniform convergence on compact sets of $\frac{\sqrt{\chi_R}}{\sqrt{q_\tau^\epsilon}}$ we find that the above tends to zero as $\epsilon\rightarrow 0$. The second term in $BB$ is handled similarly.
    \paragraph{Left-hand side of \eqref{eq:AB}}
    Regarding the left-hand side of \eqref{eq:AB} it follows that
    \begin{equation}
        \lim_{\epsilon\rightarrow 0}\int \chi_R q_{\tau_1}^\epsilon\log\frac{q_{\tau_1}^\epsilon}{\pi_{\tau_1}}\dd x = \int \chi_R q_{\tau_1}\log\frac{q_{\tau_1}}{\pi_{\tau_1}}\dd x
    \end{equation}
    again by uniform convergence of $q^\epsilon$ on compact sets and appropriate boundedness rendering the logarithm Lipschitz continuous. We therefore find 
    \begin{equation}
        \begin{aligned}
            \int &\chi_R q_{\tau_2}\log\frac{q_{\tau_2}}{\pi_{\tau_2}}\dd x - \int \chi_R q_{\tau_1}\log\frac{q_{\tau_1}}{\pi_{\tau_1}}\dd x\\
            =& -\iint ((-q_\tau\partial_{x_i}\log p_\tau) + \partial_{x_i}q_\tau) \partial_{x_i}\chi_R \log\frac{q_\tau}{\pi_\tau}\d x\d \tau\\
            &-\iint ((-q_\tau\partial_{x_i}\log p_\tau) + \partial_{x_i}q_\tau) \chi_R\partial_{x_i} \log\frac{q_\tau}{\pi_\tau}\d x \d \tau\\
            & -\iint ((-q_\tau\partial_{x_i}\log p_\tau) + \partial_{x_i}q_\tau) \partial_{x_i} \chi_R\d x\d \tau - \iint \chi_Rq\tau \partial_\tau \log \pi_\tau\d \tau \d x.
        \end{aligned}
    \end{equation}
    Using Lebesgue's dominated convergence theorem and letting $R\rightarrow \infty$ yields
    \begin{equation}
        \begin{aligned}
            \KL(\mu_{\tau_2}|\pi_{\tau_2})-& \KL(\mu_{\tau_1}|\pi_{\tau_1})\\
            &= -\iint ((-q_\tau\partial_{x_i}\log p_\tau) + \partial_{x_i}q_\tau) \partial_{x_i}\log\frac{q_\tau}{\pi_\tau}\d x\d \tau - \iint q_\tau \partial_\tau \log \pi_\tau\d \tau \d x\\
            &= -\iint (q_\tau\partial_{x_i}\log \frac{q_\tau}{p_\tau}) \partial_{x_i}\log\frac{q_\tau}{\pi_\tau}\d x\d \tau - \iint q_\tau \partial_\tau \log \pi_\tau\d \tau \d x\\
            &= -\iint |q_\tau\partial_{x_i}\log \frac{q_\tau}{p_\tau}|^2\d x\d \tau - \iint q_\tau \partial_\tau \log \pi_\tau\d \tau \d x.
        \end{aligned}
    \end{equation}
    From the above it is also apparent that 
    \[
        \tau \mapsto \KL(\mu_\tau|\pi_\tau)
    \]
    is absolutely continuous as the integral of an $L^1$ function. Consequently, $\KL(q_\tau|p_\tau)$ is a.e. differentiable with respect to $\tau$ and we have 
    \begin{equation}
        \begin{aligned}
            \frac{\d}{\d \tau}\KL(q_{\tau}|\pi_{\tau})= -\int |q_\tau\partial_{x_i}\log \frac{q_\tau}{p_\tau}|^2\d x - \int q_\tau \partial_\tau \log \pi_\tau \d x.
        \end{aligned}
    \end{equation}
    concluding the proof.
\end{proof}

\printbibliography

\end{document}

%% file: shared.tex
\usepackage{lipsum}
\usepackage{amsfonts}
\usepackage{graphicx}
\graphicspath{ {images_paper} }
\usepackage{amssymb,amsmath,amsthm}
\usepackage{epstopdf}
\usepackage{algorithmic}
\usepackage{mathtools}
\usepackage{bbm}
\usepackage{booktabs}
\usepackage{todonotes}
\usepackage{multirow}
\usepackage{makecell}
\usepackage{setspace}
\usepackage[export]{adjustbox}
\usepackage{cases}
\usepackage[createShortEnv]{proof-at-the-end}
\usepackage[shortlabels]{enumitem}
\usepackage{hyperref}
\usepackage{cleveref}
\usepackage{glossaries-extra}
\ifpdf
  \DeclareGraphicsExtensions{.eps,.pdf,.png,.jpg}
\else
  \DeclareGraphicsExtensions{.eps}
\fi

\usepackage{subcaption}
\usepackage{tikz}
\usetikzlibrary{intersections}
\usetikzlibrary{angles, quotes}
\newcommand{\R}{\mathbb{R}}
\newcommand{\1}{\mathbbm{1}}

\renewcommand{\d}{\mathrm{d}}

\newcommand{\expo}[1]{\exp\left(#1\right)}

\newcommand{\dd}{\mathrm{d}}

\DeclareMathOperator*{\KL}{KL}

\newcommand{\cf}{\textit{cf.}}

\renewcommand{\phi}{\varphi}
\renewcommand{\epsilon}{\varepsilon}

\newcommand{\dissi}{a}
\newcommand{\dissii}{b}

\newabbreviation{lsi}{LSI}{log-Sobolev inequality}
\newabbreviation{sde}{SDE}{stochastic differential equation}
\newabbreviation{pde}{PDE}{partial differential equation}
\newabbreviation{fpe}{FPE}{Fokker-Planck equation}

\usepackage{enumitem}
\setlist[enumerate]{leftmargin=.5in}
\setlist[itemize]{leftmargin=.5in}

\usepackage{amsopn}

\newtheorem{assumption}{Assumption}
\newtheorem{theorem}{Theorem}